\documentclass[a4paper]{amsart}
\usepackage{indentfirst}
\usepackage[T1]{fontenc}
\usepackage{graphicx}
\usepackage{verbatim}
\usepackage{enumerate}
\usepackage{amsthm}
\usepackage{pdfsync}

\newtheorem{theorem}{Theorem}

\newtheorem{lemma}[theorem]{Lemma}

\newtheorem{corollary}[theorem]{Corollary}

\newtheorem{remark}{Remark}

\newcommand{\E}[1]{\ensuremath{\mathbb{E}\left[#1\right]}}
\newcommand{\indicator}[1]{{\mathbf 1}_{\left\{ {#1} \right\} }}
\newcommand{\stirling}[2]{\left\{\mbox{\hspace{-1.5mm}}\begin{array}{c}#1\\
      #2\end{array}\mbox{\hspace{-1.5mm}}\right\}}
\newcommand{\R}{{\mathbf R}}
\def\d{{\text{ d}}}
\begin{document}

\title{On the One dimensional Poisson Random Geometric Graph} \author{L. Decreusefond }
\email{laurent.decreusefond@telecom-paristech.fr} \author{E. Ferraz}
\email{eduardo.ferraz@telecom-paristech.fr} \address{Institut Telecom, Telecom
  ParisTech, CNRS LTCI, Paris, France} 

\begin{abstract}
Given a Poisson process on a bounded interval, its  random geometric
graph is the graph
whose vertices are the points of the Poisson process and edges exist
between two points if and only if their distance is less than a
 fixed given threshold. We compute explicitly  the distribution of the number
 of connected components of this graph. The proof relies on inverting
 some Laplace transforms.
\end{abstract}

\maketitle

\section{Motivation}

As technology goes on \cite{kahn,lewis,pottie}, one can expect a wide
expansion of the so-called \textit{sensor networks}.  Such networks
represent the next evolutionary step in building, utilities,
industrial, home, agriculture, defense and many other
contexts~\cite{chong}.

These networks are built upon a multitude of small and cheap sensors
which are devices with limited transmission capabilities. Each sensor monitors a region around itself by measuring some
environmental quantities (e.g., temperature, humidity), detecting
intrusion, etc, and broadcasts its collected informations to other
sensors or to a central node. The question of whether information can
be shared among the whole network  is then of crucial importance.

Many researches have recently been dedicated to this problem
considering a variety of situations. It is possible to categorize
three main scenarios: those where it is possible to choose the
position of each sensor, those where sensors are arbitrarily deployed
in the target region with the control of a central station and those
where the sensor locations are random in a decentralized system.

The problem of the first scenario is that, in many cases, placing the
sensors is impossible or implies a high cost. Sometimes this
impossibility comes from the fact that the cost of placing each sensor
is too large and sometimes the network has an inherent random behavior
(like in the ad hoc case, where users move). In addition, this policy
cannot take into account the configuration of the network in the case
of failure of some sensor.

The drawback of the second scenario is a higher unity cost of sensors,
since each one has to communicate with the central station. Besides,
the central station itself increases the cost of the whole
system. Moreover, if sensors are supposed to know their positions, an
absolute positioning system has to be included in each sensor, making
their hardware even more complex and then  more expensive.


It is thus  important to investigate
the third scenario: randomly located sensors, no central
station. Actually, if we can predict some characteristics of the topology of a random network, the number of
sensors (or, as well, the power supply of them) can  be \textit{a
  priori} determined such that a given network may operate  with high
probability. For instance, we can choose the mean number of sensors
such that, if they are randomly deployed, there is more than 99\% of
probability the network to be completely connected.  

Usually, sensors are deployed in the plane or
in the ambient space, thus mathematically speaking, one has to deal
with configurations in $\R^2,$ $\R^3$ or a manifold. The recent works of Ghrist and his
collaborators~\cite{ghrist,silvacontrol}  show how, in any dimension, algebraic topology can be used to
compute the coverage of a given configuration of sensors. Trying  to
pursue their work for random settings, we quickly realized that the dimension  of the
ambient space played a key role. We then first began by the analysis of
dimension $1$, which appeared  as the most simple situation. There is here no need of the sophisticated tools of
algebraic topology. However, it doesn't seem that the problem of coverage
on a finite length interval has already been solved in the full extent
we do here. Higher dimensions will be the object of forthcoming
papers.

We
here address the situation where the radio communications are
sufficently polarized so that we can consider we have some privileged dimension.
Random coverage in one dimension has been already studied in different
contexts. Some years ago already, several analysis were done on the
circle (\cite{andrew,MR557459} and references therein) for a fixed
number of points and uniform distribution of points over the
circle. The question addressed was that of full coverage. More
recently, in~\cite{kumar},  efficient algorithms to
determine whether a region is covered considering the sensors are
deployed over a circle and distributed as a Poisson point process are given. In
\cite{andrew,noori}, the distribution of a fixed number of clusters
(see below for the definition) is given. In \cite{manohar}, sensors
are actually placed in a plan, have a fixed radius of observation. The
trace of the covered regions over a line is then studied.

Our main result is the distribution of
the number of connected components for a Poisson distribution of
sensors in a bounded interval. Our method is very much related to
queueing theory. Indeed, clusters, i.e., sequence of neighboring
sensors, are the strict analogous of busy periods. As will appear
below, our analysis turns down to be that of an M/D/1/1 queue with
preemption: when a customer arrives during a service, it preempts the
server and, since there is no buffer, the customer who was in service
is removed from the queuing system. To the best of our knowledge, such
a system has never been studied but the usual methods of Laplace
transform, renewal processes, work perfectly and with a bit of
calculus, one can compute all the characteristics we are interested
in.

The paper is organized in the following way: Section II presents the
physical and  random assumptions and defines the relevant quantities to be
calculated. The calculations and analytical results are presented in
 Section III. In section IV, two other scenarios are
presented, considering the number of incomplete clusters and clusters
placed in a circle. In Section V,  numerical examples are presented
and analyzed.

\section{Problem Formulation}

Let $L>0$ be the length of the domain in which sensors are located. We
assume that sensors are distributed according to a Poisson process of
intensity $\lambda$. Let $(X_i,\, 1\le i \le n)$ be the positions of
the sensors. We thus know that the random variables, $\Delta
X_i=X_{i+1}-X_i$ are i.i.d. and exponentially distributed. Due to
their technological limitations, each sensor can communicate only with
other sensors within a range $\epsilon$: two sensors, located
respectively at $x$ and $y$, are said to be \textit{directly
  connected} whenever $|x-y|\le \epsilon.$ For $i<j$, two sensors
located at $X_i$ and $X_j$ are indirectly connected if $X_l$ and
$X_{l+1}$ are directly connected for any $l=i,\, \cdots,\, j-1.$ A set
of sensors directly or indirectly connected is called a
\textit{cluster} and the connectivity of the whole network is measured
by the number of clusters.

The number of points in the
interval $[0,x]$ is denoted by
$N_x=\sum_{n=0}^{\infty}\indicator{X_n\leq x}$. The random variable
$A_i$ given by
\begin{eqnarray*}
  A_i=\left\{
    \begin{array}{ll}
      X_1&\mbox{ if }i=1,\\
      \mbox{inf}_{j}\{X_j|X_j>A_{i-1},X_j-X_{j-1}>\epsilon \}&\mbox{ if }i>1,
    \end{array}
  \right.
\end{eqnarray*}
represents the beginning of the $i$-th cluster, denoted by $C_i$. In
the same way, the end of this same cluster, $E_i$, is defined by
\begin{eqnarray*}
  E_i=\mbox{inf}_{j}\{X_j+\epsilon|X_{j}>A_i,X_{j+1}-X_{j}>\epsilon \}.
\end{eqnarray*}
So, the $i$-th cluster, $C_i$, has a number of points given by
$N_{E_i}-N_{A_i}$.  We define the
length $B_i$ of $C_i$ as $E_i-A_i$. The intercluster size, $D_i$, is
the distance between the end of $C_i$ and the beginning of $C_{i+1}$,
which means that $D_i=A_{i+1}-E_i$ and $\Delta A_i$ is the distance
between the first points of two consecutive clusters $C_i$, given by
$\Delta A_i=A_{i+1}-A_{i}=B_i+D_i$. 

\begin{remark}
  With this set of assumptions and definitions, we can see our problem
actually as an $M/D/1/1$ preemptive queue, Fig.~\ref{fig: queue}. In this non-conservative
system, the service time is deterministic and given by
$\epsilon$. When a customer arrives during a service, the served
customer is removed from the system and replaced by the arriving
customer. Within this framework, a cluster corresponds to what is
called a busy period, the intercluster size is the idle time  and $A_i+D_i$ is the length of the $i$-th cycle.

\begin{figure}[!ht]
\begin{center}
\includegraphics*[width=\textwidth]{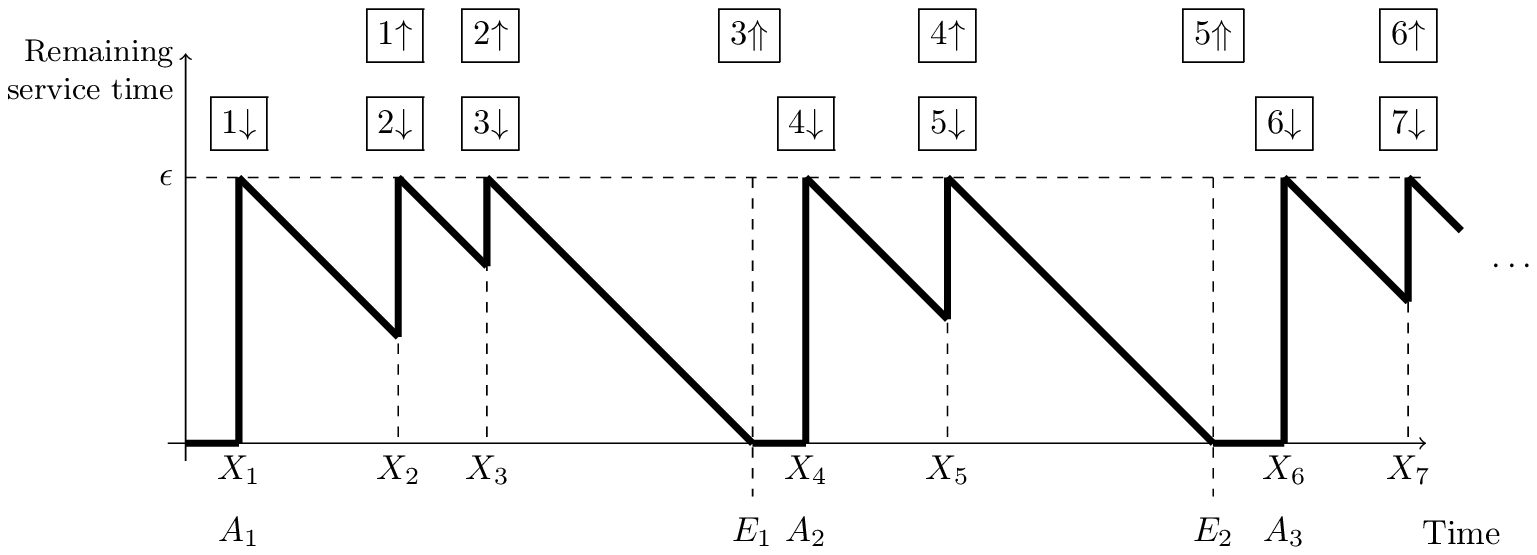}
\end{center}
  \caption{Queueing representation of the proposed problem. A down
    arrow denotes that user $i$ starts to be served. An up arrow
    indicates that user $i$ leaves the system without have finished
    the service. A double up arrow illustrates that the service of user
    $i$ finishes. It is also shown the beginning and the end of the
    $i$th busy period, respectively, $A_i$ and $E_i$.}
   \label{fig: queue}
 \end{figure}

\end{remark}

The number of complete clusters in
$[0,L]$ corresponds to the number of connected components $\beta_0(L)$
(since in dimension $1$, it coincides with the Euler characteristics
of the union of intervals, see \cite{Ghrist:2005fr,ghrist}) of
the network. The 
distance between the beginning of the first cluster and the beginning
of the $(i+1)$-th one is defined as $U_i=\sum_{k=1}^{i}\Delta
A_k$.  We also define $\Delta X_0=D_0=X_1$. Fig.~\ref{line_def}
illustrates these definitions.
\begin{figure}[!ht]
  \centering \scalebox{0.4}{\includegraphics{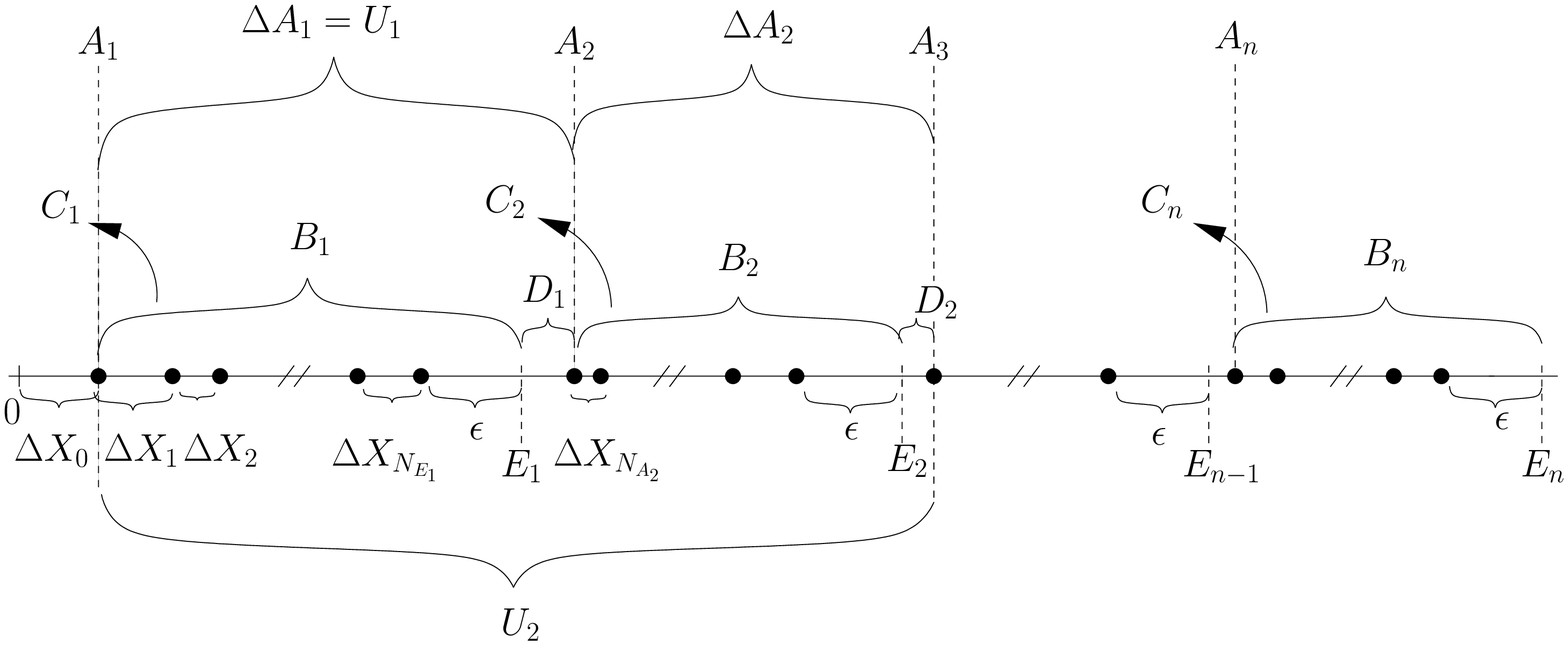}}
  \caption{Definitions of the relevant quantities of the network:
    distance between points, distance between clusters, the size of
    clusters, the size interclusters, the beginning of clusters and
    the end of clusters.}
  \label{line_def}
\end{figure}

\begin{lemma}
  \label{lemma: stopping times}
  For any $i\in\mathbb{N}_+^*$, $A_i$ and $E_i$ are stopping times.
\end{lemma}
\begin{proof}
  Let us consider the filtration $\mathcal{F}_t=\sigma\{N_a,a\leq
  t\}$. For $i=1$, we have
  \begin{eqnarray*}
    \{A_1\leq t\}\Leftrightarrow\{X_1\leq t\}\Leftrightarrow \{N_t\geq 1\}\in \mathcal{F}_t.
  \end{eqnarray*}
  Thus, $A_1$ is a stopping time. For $A_2$, we have
  \begin{eqnarray*}
    \{A_2> t\}\Leftrightarrow\bigcup_{n\geq1}\left\{N_t=n, \bigcup_{j=1}^n\left\{\Delta X_j\geq\epsilon,\bigcup_{k=j+1}^n\{\Delta X_k\leq\epsilon\} \right\}\right\}\in \mathcal{F}_t,
  \end{eqnarray*}
  so $A_2$ is also a stopping time. We proceed along the same line for
  others $A_i$ and as well for $E_i$ to prove that they are stopping
  times.
\end{proof}
Since $N$ is a strong Markov process, the next corollary is immediate.
\begin{corollary}
  \label{cor: independence}
  The set $\{B_i,D_i\, |\, i=\ge 1\}$ is a set of independent random
  variables. Moreover, $D_i$ is distributed as an exponential random
  variable with mean $1/\lambda$ and the random variables
  $\{B_i\, |\, i\ge 1\}$ are i.i.d.
\end{corollary}

\section{Calculations}
\label{calculations}
\begin{theorem}
  \label{cor: laplace d}
  The Laplace transform of the distribution of $B_i$, is given by
  \begin{equation*}
    \E{e^{-sB_i}}=\frac{\lambda+s}{\lambda + se^{(\lambda+s)\epsilon}}\cdotp\label{eq: laplace d}
  \end{equation*}
\end{theorem}
\begin{proof}
  Since $\Delta X_j$ is an exponentially distributed random variable,
  \begin{equation*}
    \E{e^{-s\Delta X_j}\indicator{\Delta X_j\leq\epsilon}}=\int_0^{\epsilon}e^{-st}\lambda e^{-\lambda t}dt 
    =\frac{\lambda}{s+\lambda}\left(1-e^{-(s+\lambda)\epsilon}\right).
  \end{equation*}
   Hence, the Laplace transform of the
  distribution of $B_1$ is given by
  \begin{align*}
    \E{e^{-sB_1}}&=\sum_{n=1}^{\infty}\E{e^{-sB_1},N_{E_1}=n}\nonumber\\
    &=\sum_{n=1}^{\infty}\E{e^{-s(\sum_{j=1}^{n-1}\Delta X_j+\epsilon)}\indicator{\Delta X_n>\epsilon}\prod_{j=1}^{n-1}\indicator{\Delta X_j\leq \epsilon}}\nonumber\\
    &=\sum_{n=1}^{\infty}\left(\E{e^{-s\Delta X_1}\indicator{\Delta X_1\leq\epsilon}}\right)^{n-1}\E{e^{-s\Delta X_n}\indicator{\Delta X_n>\epsilon}} e^{-s\epsilon}\nonumber\\
    &=\sum_{n=0}^{\infty}\left(\frac{\lambda}{s+\lambda}(1-e^{-(s+\lambda)\epsilon})\right)^{n}e^{-s\lambda} e^{-s\epsilon} \nonumber\\
    &=\frac{\lambda+s}{se^{\lambda\epsilon}e^{s\epsilon}+\lambda},
  \end{align*}
Using Collorary~\ref{cor: independence}, we have $\E{e^{-sB_1}}=\E{e^{-sB_i}}$, which concludes the proof.
\end{proof}
From this result, we can immediately calculate the Laplace transform
of the distribution of $\Delta A_i$.  Since $\Delta A_i=B_i+D_i$, we
have $\E{e^{-s\Delta A_i}}=\E{e^{-s(B_i+D_i)}}$ and using
Corollary~\ref{cor: independence}:
\begin{equation*}
  \E{e^{-s\Delta A_i}}=\E{e^{-sB_i}}\E{e^{-sD_i}}=\frac{\lambda}{\lambda+se^{(\lambda+s)\epsilon}}\cdotp
\end{equation*}
\begin{corollary}
  \label{cor: laplace delta}
  The Laplace transform of the distribution of $U_n$, for $n\geq0$ is
  given by
  \begin{equation*}
    \E{e^{-sU_n}}=\frac{\lambda^n}{\left(\lambda+ se^{(\lambda+s)\epsilon}\right)^n}\cdotp\label{eq: laplace delta}
  \end{equation*}
\end{corollary}
\begin{proof}
  We use Corollaries~\ref {cor: independence} and~\ref{cor: laplace d}
  to calculate the Laplace transform of the distribution of $U_n$,
  since $U_n=\sum_{i=1}^{n}(B_i+D_i)$:
  \begin{align*}
    \E{e^{-sU_n}}&=\prod_{i=1}^{n}\E{e^{-s B_i}}\E{e^{-sD_i}} \\
    &=\left(\frac{\lambda+s}{\lambda +
       se^{(\lambda+s)\epsilon}}\right)^n \left(\frac{\lambda}{\lambda+s}\right)^n,
  \end{align*}
hence the result.
\end{proof}
Let us define the function $p_n$ as $$p_n\, :\, x\in \R^+\longmapsto p_n(x) = \Pr(\beta_0(x)=n),$$ i.e., $p_n(x)$ is the
probability of having $n$ clusters in the interval $[0,\, x]$. Since for
all $x\in\R_+$, $0\leq p_n(x)\leq 1$, the Laplace transform of
$p_n$ with respect to $x$,
\begin{eqnarray*}
  \mathcal{L}\{p_n\}(s)=\int_0^{\infty}e^{-s x}p_n(x)\d x,
\end{eqnarray*}
is well defined.
\begin{theorem}
  \label{th: laplace_probs}
  For any  $n\geq0$, the Laplace transform of  $p_n$ is
  given by
  \begin{equation}
    \label{eq: laplace_prob}
    \mathcal{L}\{ p_n \}(s)=\frac{\lambda^n \, e^{(\lambda+s)\epsilon}}{\left(se^{(\lambda+s)\epsilon}+\lambda\right)^n}\cdotp
  \end{equation} 
\end{theorem}
\begin{proof}
  We note that, see Figure \ref{condition},
  \begin{equation*}
    \{\beta_0(x)\geq n\}\Longleftrightarrow 
    \begin{cases}
      \{\Delta X_0+U_{n-1}+B_n\leq L\}&\text{if }n\geq1,\\
        \{\Delta X_0<\infty\}& \text{if }n=0.
    \end{cases}
  \end{equation*}

 \begin{figure}[!ht]
    \centering \scalebox{0.32}{\includegraphics{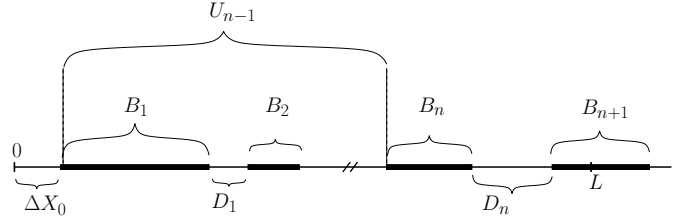}}
    \caption{Illustration of the condition equivalent to $\beta_0\geq
      n$. }
    \label{condition}
  \end{figure}

Hence  
\begin{equation*}
    \Pr(\beta_0(x)= 0)=1-\Pr(\Delta X_0+B_{1}\leq x),
  \end{equation*} 
  and
  \begin{multline}
    \label{eq: events}
    \Pr(\beta_0(x)= n)=\Pr(\Delta X_0+U_{n-1}+B_n\leq x)-\Pr(\Delta X_0+U_{n}+B_{n+1}\leq x).
  \end{multline} 
  Let
  \begin{eqnarray*}
    Y_n = \left\{\begin{array}{ll}
        \Delta X_0+U_{n-1}+B_n&\mbox{if }n\geq1\\
        0&\mbox{if }n=0
      \end{array}\right.,
  \end{eqnarray*}
 then we have:
  \begin{align}
    \mathcal{L}\{\Pr(Y_n\leq \cdot) \}(s)&=\int_0^{\infty}\Pr(Y_n\leq x)e^{-sx}dx\nonumber\\
    &=\int_0^{\infty}\int_0^{x} dP_{Y_n}(y) e^{-sx}dx\notag\\
    &=\frac{1}{s}\E{e^{-sY_n}}\notag\\
    &=\frac{1}{s}\E{e^{-s\Delta X_0}}\E{e^{-sU_{n-1}}}\E{e^{-sB_n}} \notag\\ 
&=\frac{1}{s}\frac{\lambda^n}{\left(e^{\lambda\epsilon}se^{s\epsilon}+\lambda\right)^n},
    \label{eq: laplace_int}
  \end{align}
  for $n\geq1$, where we used Corollary~\ref{cor: independence} in the
  third line. For $n=0$, the Laplace transform is trivial and given by
  $\mathcal{L}\{\Pr(Y_0\leq \cdot)\}(s)=1/s$. Substituting
  Eq.~\eqref{eq: laplace_int} in the Laplace transform of both sides
  of Eq.~\eqref{eq: events} yields:
  \begin{eqnarray*}
    \mathcal{L}\{ p_n \}(s)&=& \mathcal{L}\{\Pr(Y_n\leq \cdot)\}(s)-\mathcal{L}\{\Pr(Y_{n+1}\leq \cdot)\}(s) \nonumber \\
    &=&\frac{e^{\epsilon\lambda}e^{\epsilon s}\lambda^{n}}{\left(e^{\epsilon\lambda}se^{\epsilon s}+\lambda\right)^{n+1}},\ \ n\geq0.
  \end{eqnarray*} 
  The proof is thus complete. 
\end{proof}

\begin{lemma}
  \label{lemma: equivalence}
  Let $m$ be an positive integer. For any $x>0$, when $\epsilon\rightarrow0$,
  $\E{\beta_0^m}\rightarrow\E{N_L^m}$.
\end{lemma}
\begin{proof}
Since there is almost surely a finite number of points in $[0,\, x]$,
for almost all sample-paths, there exists $\eta>0$ such that
$\Delta X_j\ge \eta$ for any $j=1,\, \cdots,\ N_x$.  Hence for
$\epsilon<\eta$, $\beta_0(x)=N_x$. This implies that $\beta_0(x)$
tends almost surely to $N_x$ as $\epsilon$ goes to $0$. Moreover, it is
immediate by the very definition of $\beta_0(x)$ that $\beta_0(x)\le
N_x$. Since for any $m$, $\E{N_x^m}$ is finite, the proof follows by dominated convergence.
\end{proof}

Let $\mbox{Li}_{t}(z)$, $z,t\in\R$, $z<1$, be the
polylogarithm function with parameter $t$, defined by
\begin{eqnarray*}
  \mbox{Li}_{t}(z) =  \sum_{k=1}^{\infty}\frac{z^k}{k^t}\cdotp
\end{eqnarray*}
For $m$ a positive integer, consider the function of $x$
\begin{equation}
\label{eq: sum}
  M_{\beta_0}^m\, :\, x\longmapsto \E{\beta_0^m(x)}=\sum_{i=0}^{\infty}i^mp_i(x).
\end{equation}
Its Laplace transform is given by:
\begin{eqnarray*}
  \mathcal{L}\left\{M_{\beta_0}^m\right\}(s)=\int_{0}^{\infty}\E{\beta_0(x)^m}e^{-sL}\,
  dx.
\end{eqnarray*}
\begin{corollary}
  \label{cor: laplace_moments}
  Let $\alpha$ be defined as follows:
  \begin{eqnarray*}
    \alpha = \frac{e^{\epsilon\lambda}}{\lambda}se^{\epsilon s}.
  \end{eqnarray*}
  The Laplace transform of the $m$-th moment of $\beta_0(L)$ is:
  \begin{eqnarray}
    \label{eq: laplace_moments}
    \mathcal{L}\left\{M_{\beta_0}^m\right \}(s)=\frac{\alpha}{s\left(\alpha+1\right)}\mbox{Li}_{-m}\left(\frac{1}{\alpha+1}\right),
  \end{eqnarray}
  which converges, provided that $\alpha >0.$
\end{corollary}
\begin{proof}
  Applying the Laplace transform of both sides of Eq.~\eqref{eq: sum},
  we get:
  \begin{align*}
    \mathcal{L}\left\{M_{\beta_0}^m\right \}(s)&=\sum_{i=1}^{\infty}i^m\mathcal{L}\{p_i\}(s)\\
    &=\frac{\frac{e^{\epsilon\lambda}}{\lambda}e^{\epsilon s}}{\left(\frac{e^{\epsilon\lambda}}{\lambda}se^{\epsilon s}+1\right)}\sum_{i=1}^{\infty}\frac{i^m}{\left(\frac{e^{\epsilon\lambda}}{\lambda}se^{\epsilon s}+1\right)^i}\\
    &=\frac{\alpha }{s\left(\alpha +1\right)}\mbox{Li}_{-m}\left(\frac{1}{\alpha +1}\right),
  \end{align*}
  concluding the proof.
\end{proof}
We  define $\tiny{\stirling{m}{k}}$ as the Stirling
number of  second kind~\cite{graham}, i.e.,
$\tiny{\stirling{m}{k}}$ is the number of ways to partition a set of
$m$ objects into $k$ groups. They are intimately related to
polylogarithm by the following identity (see~\cite{wood}) valid for any positive
  integer $m,$
  \begin{equation}\label{eq:1}
    \mbox{Li}_{-m}(z)=\sum_{k=0}^{m}\frac{(-1)^{m+k}k!\stirling{m+1}{k+1}}{(1-z)^{k+1}}\cdotp
  \end{equation}
\begin{corollary}
  \label{cor: moments}
  The $m$-th moment of the number of clusters on the interval $[0,L]$
  is given by:
  \begin{eqnarray}
    \label{eq: result_moments}
    M_{\beta_0}^m(L)=\sum_{k=1}^{m}\stirling{m}{k}\left(\frac{L}{\epsilon}-k\right)^k \left(\lambda\epsilon e^{-\epsilon\lambda}\right)^k\indicator{L/\epsilon>k}.
  \end{eqnarray}
\end{corollary}
\begin{proof}
  Using \eqref{eq:1} in the result of Corollary~\ref{cor: laplace_moments}, we get:
  \begin{align*}
    \mathcal{L}\left\{M_{\beta_0}^m\right \}(s)&=\frac{\alpha}{s}\sum_{k=0}^{m}\frac{(-1)^{m+k}k!\stirling{m+1}{k+1}(1+\alpha)^k}{\alpha^{k+1}\left(\alpha+1\right)}\\
    &=\frac{1}{s}\sum_{k=0}^{m}c_{k,m}\frac{1}{\alpha^k},
  \end{align*}
  where the coefficients $c_{k,m}$ are integers given by:
  \begin{eqnarray*}
    c_{k,m}=\sum_{j=k}^{m}(-1)^{j}j!\stirling{m+1}{j+1}{j \choose k}.
  \end{eqnarray*}
  Using the following identity of Stirling numbers~\cite{roman},
  \begin{eqnarray*}
    \sum_{j=0}^{m}(-1)^{j}j!\stirling{m+1}{j+1}=0,
  \end{eqnarray*}
  we find that $c_{0,m}=0$ for $m$ a positive integer. So we can write
  the Laplace transform of the moments as
  \begin{equation*}
    \label{eq: laplace_moments2}
    \mathcal{L}\left\{M_{\beta_0}^m\right \}(s)=\sum_{k=1}^{m}c_{k,m}\frac{\left(\lambda e^{-\epsilon\lambda}\right)^k}{s^{k+1}e^{ks\epsilon}}
  \end{equation*}
  and apply the inverse of the Laplace transform in both size of
  Eq.~\eqref{eq: laplace_moments2} to obtain:
  \begin{align*}
    \label{eq: laplace_moments3}
    M_{\beta_0}^m(L)&=\mathcal{L}^{-1}\left\{\sum_{k=1}^{m}c_{k,m}\frac{\left(\lambda e^{-\epsilon\lambda}\right)^k}{s^{k+1}e^{ks\epsilon}}\right\}(L)\nonumber\\
    &=\sum_{k=1}^{m}c_{k,m}\left(\lambda e^{-\epsilon\lambda}\right)^k\mathcal{L}^{-1}\left\{\frac{1}{s^{k+1}e^{ks\epsilon}}\right\}(L)\nonumber\\
    &=\sum_{k=1}^{m}\frac{c_{k,m}}{k!}(L-k\epsilon)^k \left(\lambda e^{-\epsilon\lambda}\right)^k\indicator{L>k\epsilon}
  \end{align*}
  According to Lemma~\ref{lemma: equivalence}, when
  $\epsilon\rightarrow0$, we obtain
  \begin{eqnarray*}
    M_{\beta_0}^m(L)=\E{N_L^{m}}=\sum_{k=1}^{m}\frac{c_{k,m}}{k!}(L\lambda)^k\indicator{L>0}.
  \end{eqnarray*}
  Hence, for any $\lambda>0$,
  \begin{eqnarray*}
    \sum_{k=1}^{m}\frac{c_{k,m}}{k!}(L\lambda)^k\indicator{L>0}=\sum_{k=1}^{m}\stirling{m}{k}(L\lambda)^k\indicator{L>0},\end{eqnarray*}
  which shows that
  \begin{eqnarray*}
    c_{k,m}=\stirling{m}{k}k!\ .
  \end{eqnarray*}
  Thus, we have proved~\eqref{eq: result_moments} for any positive
  integer $m$.
\end{proof}

\begin{theorem}
  \label{th: main}
For any $n$, $L$,
  $\lambda$ and $\epsilon$, we have:
  \begin{eqnarray}
    \label{eq: main}
    \Pr(\beta_0(L)=n)=\frac{1}{n!}\sum_{i=0}^{\lfloor L/\epsilon\rfloor-n}\frac{(-1)^{i}}{i!} ((L-(n+i)\epsilon) \lambda e^{-\lambda\epsilon})^{n+i}.
  \end{eqnarray}
\end{theorem}
\begin{proof}
Since $\beta_0(L)\le N_L$ and since $\E{e^{s N_L}}$ is finite for
any $s\in \R$, 
 we have, for any $s\ge 0$:
  \begin{eqnarray*}
    \E{e^{-s\beta_0(L)}}=\sum_{k=0}^\infty (-1)^k\frac{s^k}{k!} \, \E{\beta_0^k(L)}.
  \end{eqnarray*}
  Rearranging the terms of the right-side hand and substituting
  $M_{\beta_0}^m(L)$, by the result of Eq.~\eqref{eq: result_moments},
  we obtain:
  \begin{equation*}
    \label{eq: roman}
    \E{e^{-s\beta_0(L)}}=\sum_{k=0}^{\infty}\left( (L-k\epsilon)^k\left(\lambda e^{-\lambda\epsilon}\right)^k\indicator{L>k\epsilon}
      \sum_{j=k}^{\infty}\frac{(-s)^j}{j!}\stirling{j}{k}\right)\cdotp
  \end{equation*}
Furthermore, it is known (see~\cite{roman}) that 
  \begin{equation*}
    \sum_{j=k}^{\infty}\frac{x^j}{j!}\stirling{j}{k}=\frac{1}{k!}(e^x-1)^k.
  \end{equation*}
Hence, 
  \begin{eqnarray*}
    \E{e^{-s\beta_0(L)}}=\sum_{k=0}^{\infty}(L-k\epsilon)^k\left(\lambda e^{-\lambda\epsilon}\right)^k\indicator{L>k\epsilon}\frac{(e^{-s}-1)^k}{k!}\cdotp
  \end{eqnarray*}
  By inverting the Laplace transforms, we get:
  \begin{eqnarray*}
    \sum_{k=0}^{\infty}\sum_{i=k}^{\infty}\frac{(-1)^{i}}{i!}{i\choose n} 
    \delta_{(k-n)} (k\epsilon-L)^k\left(\lambda
      e^{-\lambda\epsilon}\right)^k\indicator{L>k\epsilon}, 
  \end{eqnarray*}
where $\delta_a$ is the Dirac measure at point $a$.
  After some simple algebra, we find the expression of the probability
  that an interval contains $n$ complete clusters:
  \begin{equation*}
    \Pr(\beta_0(L)=n)=p_n(L) =\frac{1}{n!}\sum_{i=0}^{\lfloor L/\epsilon\rfloor-n}\frac{(-1)^{i}}{i!} ([L-(n+i)\epsilon] \lambda e^{-\lambda\epsilon})^{n+i},
  \end{equation*}
  concluding the proof.
\end{proof}
\begin{lemma}
  \label{lemma: props_p}
  For $x\geq0$, $p_n(x)$ has the three following properties:
  \begin{enumerate}[i)]
  \item $p_n(x)$ is differentiable;
  \item $\lim_{x\rightarrow\infty}p_n(x)=0$;
  \item $\lim_{x\rightarrow\infty}\frac{dp_n(x)}{dx}=0$.
  \end{enumerate}
\end{lemma}
\begin{proof}
  Let $j$ be a non-negative integer. The function is obviously
  differentiable when $x/\epsilon\not=j$. Besides, we have
  \begin{eqnarray*}
    \lim_{x\rightarrow \epsilon j^+}p_n(x)-\lim_{ x\rightarrow \epsilon j^-}p_n(x)=\lim_{x\rightarrow\epsilon j^+}\frac{(-1)^{j}}{j!} \left((x-(n+j)\epsilon)\frac{1}{a}\right)^{n+j}\cdotp
  \end{eqnarray*}
  Since the right-hand term function of $x$ is zero as well as its
  derivative for all $j$, the function is also derivable when
  $x/\epsilon=j$, which proves i). Items
  ii) and iii) are direct consequences of Final Value theorem in
  the Laplace transform of $p_n$ and its derivative.
\end{proof}
The expression of $p_n$ gives us a Laplace pair
between the $x$ and $s$ domains:
\begin{eqnarray}
  \label{eq: laplace_pair}
  \frac{\indicator{x\geq0}}{n!}\sum_{i=0}^{\lfloor x/\epsilon\rfloor-n}\frac{(-1)^{i}}{i!} \left((x-(n+i)\epsilon)\frac{1}{a}\right)^{n+i}
  \stackrel{\mathcal{L}}{\Longleftrightarrow}\frac{ae^{\epsilon s}}{\left(ase^{\epsilon s}+1\right)^{n+1}} .
\end{eqnarray}
We can use this relation to find the distributions of $B_i$ and $U_n$.
\begin{theorem}
  The distributions of $B_i$ and $U_n$, respectively $f_{B_i}(x)$ and
  $f_{U_n}(x)$ are
  \begin{eqnarray}
    f_{B_i}(x)=\left[\lambda e^{-\epsilon\lambda} p_0(x-\epsilon)+e^{-\epsilon\lambda} \frac{d}{dx}p_0(x-\epsilon)\right]\indicator{x>\epsilon},\label{eq: B}
  \end{eqnarray}
  and
  \begin{eqnarray}
    f_{U_n}(x)=\lambda e^{-\epsilon\lambda} p_{n-1}(x-\epsilon)\indicator{x>\epsilon},\label{eq: U}
  \end{eqnarray}
  where the expressions of $p_0(x-\epsilon)$ and
  $\frac{d}{dx}p_0(x-\epsilon)$ are straightforwardly obtained from
  Eq.~\eqref{eq: main}.
\end{theorem}
\begin{proof}
  According to Corollary~\ref{cor: laplace d}:
  \begin{eqnarray*}
    \E{e^{-sB_i}}&=&\frac{1}{\lambda}\frac{(\lambda+s)}{\frac{e^{\lambda\epsilon}}{\lambda}se^{s\epsilon}+1}\nonumber\\
    &=&\lambda e^{-\epsilon\lambda} 
    \frac{e^{\epsilon\lambda}}{\lambda}\frac{e^{\epsilon s}}{\frac{e^{\epsilon\lambda}}{\lambda}se^{\epsilon s}+1}
    e^{-\epsilon s}\nonumber
    + e^{-\epsilon\lambda} s   
    \frac{e^{\epsilon\lambda}}{\lambda}\frac{e^{\epsilon s}}{\frac{e^{\epsilon\lambda}}{\lambda}se^{\epsilon s}+1}
    e^{-\epsilon s}\nonumber\\
    &=&\lambda e^{-\epsilon\lambda} e^{-\epsilon s}\mathcal{L}\left\{p_0(\cdot)\right\}(s) + e^{-\epsilon\lambda} e^{-\epsilon s}s\mathcal{L}\left\{ p_0(\cdot)\right\}(s).
  \end{eqnarray*}
  Here, using the inverse Laplace transform established in
  Eq.~\eqref{eq: laplace_pair} and remembering that $p_0(x^-)=0$, we
  get an analytical expression for $f_{B_i}(x)$, proving
  Eq.~\eqref{eq: B}.

  Proceeding in a similar fashion, we can find the distribution of
  $U_n$ by inverting its Laplace transform given by
  Corollary~\ref{cor: laplace delta}:
  \begin{eqnarray*}
    \E{e^{-sU_n}}&=&\frac{1}{\left(\frac{e^{\lambda\epsilon}}{\lambda}se^{s\epsilon}+1\right)^n}\nonumber\\
    &=&\lambda e^{-\epsilon\lambda} 
    \frac{e^{\epsilon\lambda}}{\lambda}\frac{e^{\epsilon s}}{\left(\frac{e^{\epsilon\lambda}}{\lambda}se^{\epsilon s}+1\right)^n}
    e^{-\epsilon s}\nonumber\\
    &=&\lambda e^{-\epsilon\lambda} e^{-\epsilon s}\mathcal{L}\left\{p_{n-1}(\cdot)\right\}(s).
  \end{eqnarray*}
  We thus have Eq.~\eqref{eq: U}.
\end{proof}

We can also obtain the probability that the segment $[0,L]$ is
completely covered by the sensors. To do this, we remember that the
first point (if there is one) is capable to cover the interval
$[X_1-\epsilon, X_1+\epsilon]$. 
\begin{theorem}
  \label{th: coverage}
  Let $R_{m,n}(x)$ be defined as follows:
  \begin{eqnarray*}
    R_{m,n}(x)=\sum_{i=m}^{\lfloor x/\epsilon\rfloor-1}\left[\left( e^{-\lambda\epsilon} \right)^{i+n} \sum_{j=0}^{i+n}\frac{(\lambda[(1-i)\epsilon-x])^{j}}{j!}\right]\cdotp
  \end{eqnarray*}
  Then,
  \begin{multline}
    \label{eq: coverage}
    \Pr([0,L]\mbox{ is
      covered})=R_{0,1}(L)-e^{-\lambda\epsilon}R_{0,1}(L-\epsilon)\\-e^{-\lambda\epsilon}R_{1,0}(L)
    +e^{-2\lambda\epsilon}R_{1,0}(L-\epsilon).
  \end{multline}
\end{theorem}

\begin{proof}
  The condition of total coverage is the same as
  \begin{eqnarray*}
   \Bigl\{ \forall x\in[0,L],\exists X_i\in[0,L]\, \Bigl| x\in [X_1-\epsilon, X_1+\epsilon]\Bigr\},
  \end{eqnarray*}
  which means that:
  \begin{eqnarray*}
    \{[0,L]\mbox{ is covered} \}\Leftrightarrow \{B_1\geq L-X_1\}\cap\{X_1\leq \epsilon\}.
  \end{eqnarray*}
  Hence,
  \begin{eqnarray*}
    \Pr([0,L]\mbox{ is covered})=\int_0^{\epsilon}\Pr(B_1\geq L-X_1|X_1=x)dP_{X_1}(x),
  \end{eqnarray*}
  and since $B_1$ and $X_1$ are independent:
  \begin{eqnarray*}
    \Pr([0,L]\mbox{ is covered})=\int_0^{\epsilon}\int_{L-x}^{\infty}f_{B_1}(u)\lambda e^{-x\lambda} du dx.
  \end{eqnarray*}
 The result then follows from Lemma~\ref{lemma: props_p} and  some tedious
 but straightforward algebra.
 \end{proof}


\section{Other Scenarios}
The method can be used to
calculate $p_n$ for other definitions of the number of clusters. We consider  two other definitions: the number of
incomplete clusters and the number of clusters in a circle.

\subsection{Number of incomplete clusters}
The major difference with Sec.~\ref{calculations} is that a cluster is
now taken into account  as soon as one of the point of the cluster is inside the
interval $[0,L]$. So, for instance, in Fig.~\ref{condition}, we count
actually $n+1$ incomplete clusters. We define $\beta_0'(L)$ as the number
of incomplete clusters on an interval $[0,L]$.
\begin{theorem}
  Let $G(k)$ be defined as
  \begin{eqnarray*}
    G(k)=(-1)^{k}\left(e^{-k\lambda\epsilon}\sum_{j=0}^{k}\frac{[\lambda(k\epsilon-L)]^j}{j!}-e^{-\lambda L} \right) \indicator{T>k\epsilon}
  \end{eqnarray*}
  for $k\in\mathbb{N}_+$ and $G(-1)=e^{-\lambda L}$. Then
  \begin{equation*}
    \label{eq: main2}
    \Pr(\beta_0'(L)=n)=\sum_{i=n}^{\lfloor L/\epsilon\rfloor+1}(-1)^{i+n}{i\choose n}(G(i-1)+G(i)),\mbox{ for } n\geq0.
  \end{equation*}
\end{theorem}
\begin{proof}
  The condition of $\beta_0'(L)\geq n$ is now given by:
  \begin{equation*}
    \{\beta_0'\geq n\}\Longleftrightarrow 
    \begin{cases}
       \{\Delta X_0+U_{n-1}\leq L\}&\mbox{if }n\geq1,\\
        \{\Delta X_0<\infty\}& \mbox{if }n=0. 
    \end{cases}
  \end{equation*} 
  We define $Y_n$ as
  \begin{eqnarray*}
    Y_n = \left\{\begin{array}{ll}
        \Delta X_0+U_{n-1}&\mbox{if }n\geq1\\
        0&\mbox{if }n=0.
      \end{array}\right.
  \end{eqnarray*}
  Repeating the same calculations, we find the Laplace transform of
  $ \Pr(\beta_0'(.)=n)$:
  \begin{equation*}
    \mathcal{L}\{ \Pr(\beta_0'(\cdot)=n) \}(s)=
    \begin{cases}
      \dfrac{\lambda}{s+\lambda}\dfrac{e^{\epsilon\lambda}}{\lambda}\dfrac{e^{\epsilon s}}{\left(\frac{e^{\epsilon\lambda}}{\lambda}se^{\epsilon s}+1\right)^{n}}&\mbox{if }n\geq1,\\
        \dfrac{1}{\lambda+s}&\mbox{if }n=0.   
    \end{cases}
  \end{equation*} 
  With this expression, following the lines of Lemma~\ref{lemma:
    equivalence}, we obtain:
  \begin{eqnarray*}
    \mathcal{L}\left\{\E{\beta_0'(\cdot)^m}\right\}(s)=\sum_{k=1}^{m+1}\stirling{m+1}{k}(k-1)!\frac{1}{s^k}\frac{\lambda}{\lambda+s} \left(\frac{\lambda e^{-\lambda\epsilon}}{e^{s\epsilon}}\right )^{k-1}.
  \end{eqnarray*}
  Then, we write:
  \begin{eqnarray*}
    \frac{\lambda}{\lambda+s}\frac{1}{s^k}=\frac{(-1)^{k}}{\lambda^{k-1}}\frac{1}{\lambda+s}+\sum_{i=1}^{k}\frac{1}{s^i} \left(\frac{-1}{\lambda}\right)^{k-i},
  \end{eqnarray*}
  to find an expression with a well known Laplace transform inverse,
  and after inverting it, we obtain:
  \begin{equation*}
    \E{\beta_0'^m}=\sum_{k=0}^{m}\stirling{m+1}{k+1}k!G(k).
  \end{equation*}
  Expanding the Laplace transform of the distribution of $\beta_0'(L)$ in
  a Taylor series and rearranging terms, we get
  \begin{multline*}
    \E{e^{-s\beta_0'(L)}}=1+G(0)
    \sum_{j=1}^{\infty}\frac{(-s)^j}{j!}\stirling{j}{1}+
    \left(\sum_{k=1}^{\infty}G(k)
      \sum_{j=k}^{\infty}\frac{(-s)^j}{j!}\stirling{j+1}{k+1}\right).
  \end{multline*}
  Now, we use another  recurrence that Stirling numbers obey~\cite{roman},
  \begin{equation*}
    \stirling{j+1}{k+1}=\stirling{j}{k}+(k+1) \stirling{j}{k+1},
  \end{equation*}
 to get:
  \begin{eqnarray*}
    \sum_{j=k}^{\infty}\frac{x^j}{j!}\stirling{j+1}{k+1}&=&\sum_{j=k}^{\infty}\frac{x^j}{j!}\left(\stirling{j}{k}+(k+1) \stirling{j}{k+1}\right) \\
    &=&\frac{1}{k!}(e^x-1)^k+\frac{1}{k!}(e^x-1)^{k+1}.
  \end{eqnarray*}
Hence,
  \begin{eqnarray*}
    \E{e^{-s\beta_0'(L)}}=1+\sum_{k=1}^{\infty}(G(k-1)+G(k)) (e^{-s}-1)^k.
  \end{eqnarray*}
  Inverting this expression for any non-negative integer $n$, we have
  the searched distribution.
\end{proof}

\subsection{Number of clusters in a circle}
We investigate now the case where the points of the process are
deployed over a circumference and we want to count the number of
complete clusters, which corresponds to calculate the Euler's
Characteristic of the total coverage, so we call this quantity
$\chi$. Without loss of generality , we can choose an
arbitrary point to be the origin. 
\begin{theorem}
  \label{th: chi}
  The distribution of the Euler's Characteristic, $\chi(L)$, when the
  points are deployed over a circumference of length $L$ is given by
  \begin{multline}
    \label{eq: chi}
    \Pr(\chi(L)=n)=e^{-\lambda L}\indicator{n=0}+(1-e^{-\lambda L})\frac{\lambda e^{-\epsilon\lambda}}{n!}\sum_{i=0}^{\lfloor L/\epsilon\rfloor-n}\left[\frac{(-1)^i}{i!}\right.\\
    \left. ([L-(n+i)\epsilon]\lambda
      e^{-\epsilon\lambda})^{n+i-1}\left(L+(n+i)\left(\frac{1}{\lambda}-\epsilon
        \right) \right)\right],
  \end{multline}
  for $n\geq0$.
\end{theorem}
\begin{proof}
 If there is no points on the circle,  $\chi(L)=0$. Otherwise, if there is at least one point, we
  choose the origin at this point and we have  equivalence between
  the events:
  \begin{eqnarray*}
    \{\chi(L)\geq n\}\Leftrightarrow \left\{\begin{array}{ll}
        \{U_{n-1}+B_n\leq L\}\cap\{N_L>0\}&\mbox{if }n\geq1,\\
        \{\Delta X_0<\infty\}& \mbox{if }n=0.
      \end{array}\right.
  \end{eqnarray*}
  In Fig.~\ref{condition_circle} we present an example of this
  equivalence.

  \begin{figure}[!ht]
    \begin{center}
      \includegraphics*[width=\textwidth]{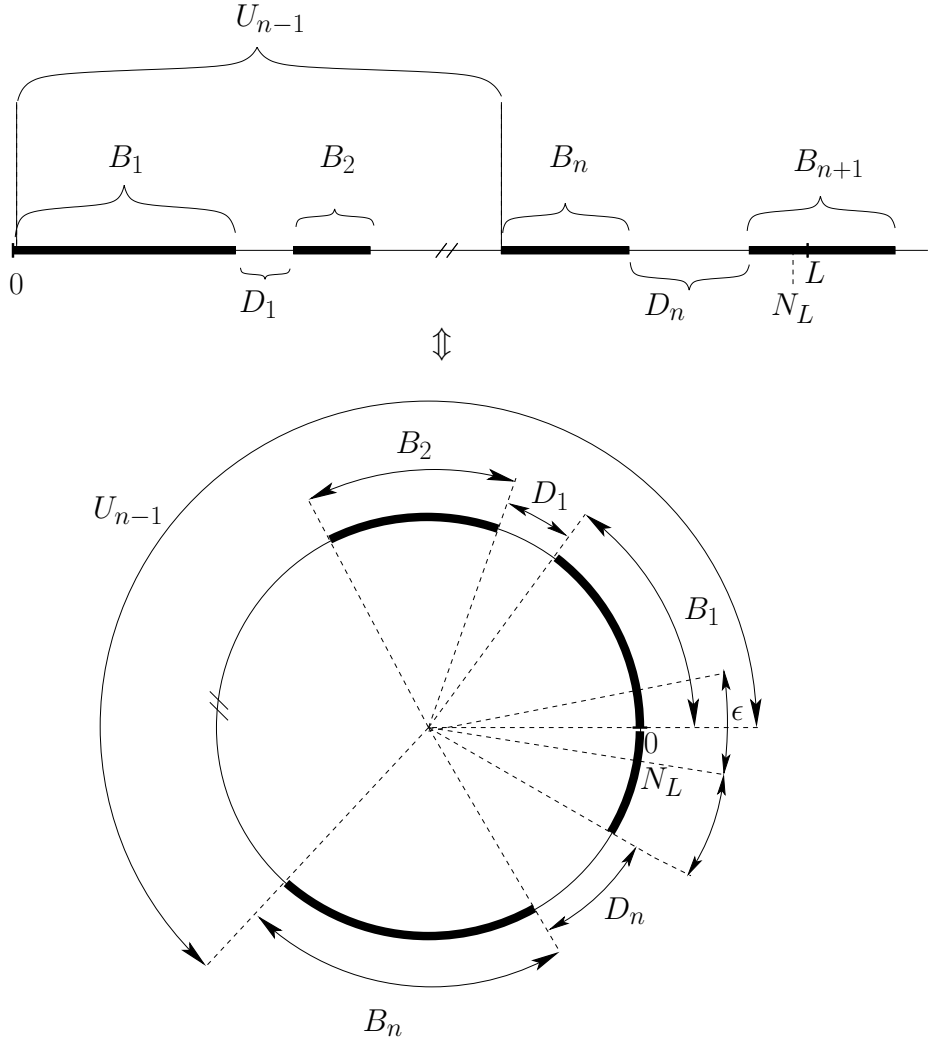}
    \end{center}
    \caption{Illustration of the condition equivalent to $\chi(L)\geq
      n$. Since the coverage of the last point on $[0,L]$ overlaps the
      cluster with a point in zero, they are actually contained in the
      same cluster }\label{condition_circle}
  \end{figure}

  We can define $Y_n$ as
  \begin{eqnarray*}
    Y_n = \left\{\begin{array}{ll}
        U_{n-1}+B_n&\mbox{if }n\geq1\\
        0&\mbox{if }n=0,
      \end{array}\right.
  \end{eqnarray*}
  to find the Laplace transform or $\Pr(\chi(L)=n)$:
  \begin{eqnarray}
    \mathcal{L}\{ \Pr(\chi(\cdot)=n) \}(s)=
    (1-e^{-\lambda L})\dfrac{\lambda+s}{\lambda}\dfrac{e^{\epsilon\lambda}}{\lambda}\dfrac{e^{\epsilon s}}{\left(\frac{e^{\epsilon\lambda}}{\lambda}se^{\epsilon s}+1\right)^{n}}\cdotp
  \end{eqnarray} 
The number of clusters is almost
  surely equal to the number of points when $\epsilon\rightarrow0$,
  so 
  \begin{multline*}
    \E{\chi(L)^m}=(1-e^{-\lambda L})\lambda e^{-\epsilon\lambda}\sum_{k=1}^{m}\left[\stirling{m}{k}([L-k\epsilon]\lambda e^{-\epsilon\lambda})^{k-1}\right.\\
    \left. \left(L+k\left(\frac{1}{\lambda}-\epsilon \right)
      \right)\indicator{L>k\epsilon}\right].
  \end{multline*}
  Expanding the Laplace transform in a Taylor series and rearranging
  terms, as we did previously, yields
  \begin{multline*}
    \E{e^{-s\chi(L)}}=(1-e^{-\lambda L})\lambda
    e^{-\epsilon\lambda}\sum_{k=0}^{\infty}\left[\left([L-k\epsilon]\lambda
        e^{-\epsilon\lambda}\right)^{k-1}\left(L+k\left(\frac{1}{\lambda}-\epsilon
        \right)\right)\right.\\\left.  \indicator{L>k\epsilon}
      \sum_{j=k}^{\infty}\frac{(-s)^j}{j!}\stirling{j}{k}\right].
  \end{multline*}
  Since
  \begin{equation*}
    \sum_{j=k}^{\infty}\frac{(-s)^j}{j!}\stirling{j}{k}=\frac{(e^{-s}-1)^k}{k!},
  \end{equation*}
  we can directly invert this Laplace transform, add the case where
  there are no points for $\chi(L)=0$, and the theorem is proved.
\end{proof}
\section{Examples}
\label{examples}
We consider some examples to illustrate the results of the
paper. Here, the behavior of the mean and the variance of $\beta_0(L)$ as
well as $Pr(\beta_0(L)=n)$ are presented.

From Eq.~\eqref{eq: result_moments}, we have that $\E{\beta_0(L)}$ is
given by:
\begin{eqnarray*}
\E{\beta_0(L)}=(L-\epsilon)\lambda e^{-\epsilon \lambda}\indicator{L>\epsilon}.
\end{eqnarray*}
This expression agrees with the intuition that there are three typical
regions given a fixed $\epsilon$.  When $\lambda$ is much smaller than
$1/\epsilon$, the number of clusters is approximatively the number of
sensors, since the connections with few sensors will unlikely happen,
which can be seen from the fact that $\E{\beta_0(L)}\rightarrow
L\lambda$ when $\lambda\rightarrow 0$. As we increase $\lambda$, the
mean number of direct connections overcomes the mean number of sensors
and, at some value of $\lambda$, we expect that $\E{\beta_0(L)}$
decreases, when adding a point is likely to connect disconnected
clusters. We remark that the maximum occurs exactly for
$\epsilon=1/\lambda$, i.e., when the mean distance between two sensors
equals the threshold distance for them to be connected. At this
maximum, $\E{\beta_0(L)}$ takes the value of
$(L/\epsilon-1)e^{-1}$. Finally, when $\lambda$ is too large, all
sensors tend to be connected and there is only one cluster which even
goes beyond $L$, so there are no complete clusters into the interval
$[0,L]$. This is trivial when we make $\lambda\rightarrow\infty$ in
the last equation. Figure~\ref{beta0_mean} shows this behavior when
$L=4$ and $\epsilon=1$.
\begin{figure}[thp]
  \centering \scalebox{.85}{\includegraphics{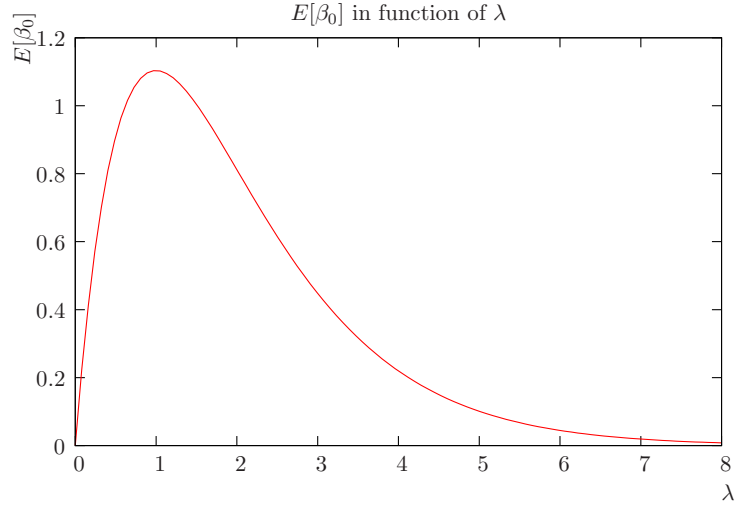}}
  \caption{Variation of the mean number of clusters in function of
    $\lambda$ when $L=4$ and $\epsilon=1$.}
  \label{beta0_mean}
\end{figure}

The variance can be obtained also by Eq.~\eqref{eq: result_moments}:
\begin{multline*}
  \mbox{Var}(\beta_0(L))=(L-\epsilon)\lambda e^{-\epsilon \lambda}\indicator{L>\epsilon}+(L-2\epsilon)\lambda^2 e^{-2\epsilon \lambda}\indicator{L>2\epsilon}\\-(L-\epsilon)^2\lambda^2 e^{-2\epsilon \lambda}\indicator{L>\epsilon},
\end{multline*}
and under the condition that $L>2\epsilon$:
\begin{eqnarray*}
  \mbox{Var}(\beta_0(L))=(L-\epsilon)\lambda e^{-\epsilon \lambda}+\epsilon(3\epsilon-2L)\lambda^2 e^{-2\epsilon \lambda}.
\end{eqnarray*}
Fig.~\ref{beta0_var}
\begin{figure}[thp]
  \centering \scalebox{0.85}{\includegraphics{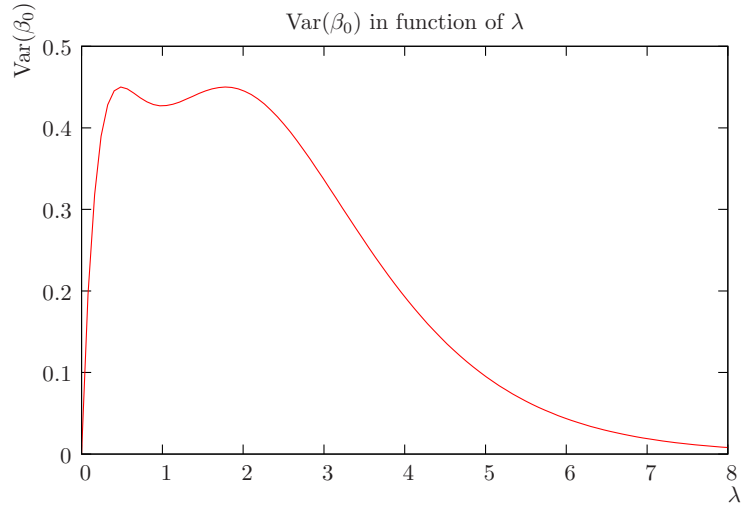}}
  \caption{Behavior of the variance of the number of clusters in
    function of $\lambda$ when $L=4$ and $\epsilon=1$.}
  \label{beta0_var}
\end{figure}
shows a plot of Var($\beta_0(L)$) in function of $\lambda$ for $L=4$ and
$\epsilon=1$.  We can expect that, when $\lambda$ is small compared to
$\epsilon$, the plot should be approximatively linear, since there
would not be too much connections in the network and the variance of
the number of clusters should be close to the variance of the number
of sensors given by $\lambda L$. Since $\beta_0(L)$ tends almost surely
to 0 when $\lambda$ goes to infinity, Var$(\beta_0(L))$ should also tend
to 0 in this case. Those two properties are observed in the
plot. Besides, we find the critical points of this function, and
again, $\lambda=1/\epsilon$ is one of them and at this value
Var$(\beta_0(L))=(L/\epsilon)e^{-1}+(3-2L/\epsilon)e^{-1}$. The other two
are the ones satisfying the transcendant equation:
\begin{eqnarray*}
  \lambda e^{-\lambda\epsilon}=\frac{L-\epsilon}{2\epsilon(2L-3\epsilon)}\cdotp
\end{eqnarray*}
By using the second derivative, we realize that $1/\epsilon$ is
actually a minimum. Besides, if $L\leq2\epsilon$, there is just one
critical point, a maximum, at $\lambda=1/\epsilon$.

The last example in the section is performed with the result obtained
in Theorem~\ref{th: main}. We consider again $L=4$ and $\epsilon=1$ to
obtain the following distributions:
\begin{eqnarray*}
  \Pr(\beta_0(L)=0)&=&1-3\lambda
  e^{-\lambda}+2\lambda^2e^{-2\lambda}-1/6\lambda^3e^{-3\lambda}, \\
  \Pr(\beta_0 (L)=1)&=&3\lambda
  e^{-\lambda}-4\lambda^2e^{-2\lambda}+1/2\lambda^3e^{-3\lambda}, \\
  \Pr(\beta_0 (L)=2)&=&2\lambda^2e^{-2\lambda}-1/2\lambda^3e^{-3\lambda}, \\
  \Pr(\beta_0 (L)=3)&=&1/6\lambda^3e^{-3\lambda}, \\
  \Pr(\beta_0 (L)>3)&=&0.
\end{eqnarray*}
Those expressions are simple and they have at most four terms, since
$L=4\epsilon$. We plot these functions in Fig.~\ref{chi_probs_1d}. The
critical points on those plots at $\lambda=1/\epsilon$ are confirmed
for the fact that, in function of $\lambda$, for every $n$,
$\Pr(\chi(L)=n)$ can be represented as a sum
\begin{eqnarray*}
  \sum_{i=0}^{j}q_{i,j}(\lambda e^{-\lambda\epsilon})^i
\end{eqnarray*}
where the coefficients $q_{i,j}$ are constant in relation to
$\lambda$. However, $(\lambda e^{-\lambda\epsilon})^i$ has a critical
point at $\lambda=1/\epsilon$ for all $i>0$, so this should be also a
critical point of $\Pr(\chi(L)=n)$. If $\lambda$ is small, we should
expect that $\Pr(\chi(L)=0)$ is close to one, since it is likely to $N$
have no points. For this reason, in this region, $\Pr(\chi(L)=n)$ for
$n>0$ is small. When $\lambda$ is large, we expect to have very large
clusters, likely to be larger than $L$, so it is unlikely to have a
complete cluster in the interval and, again, $\Pr(\chi(L)=0)$ approaches
to the unity, while $\Pr(\chi(L)=n)$ for $n>0$ become again small.
\begin{figure}[thp]
  \centering \scalebox{.85}{\includegraphics{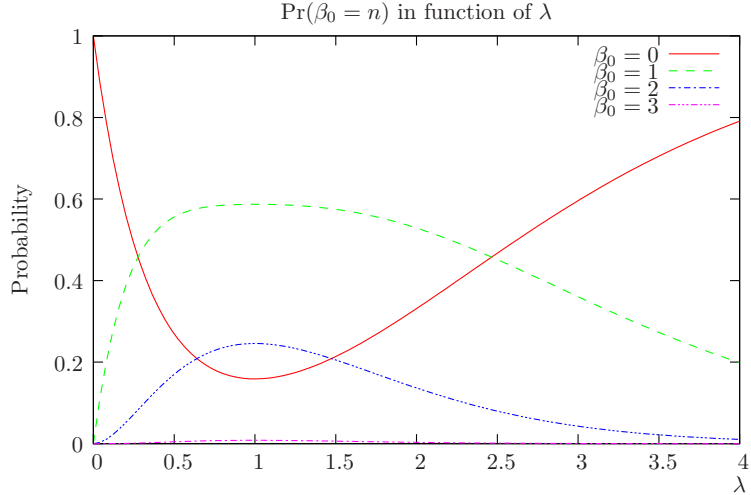}}
  \caption{Probabilities of connectiveness, $\Pr(\beta_0(L)=n)$, for
    $n=0,1,2,3$, in function of $\lambda$ when $L=4$ and
    $\epsilon=1$.}
  \label{chi_probs_1d}
\end{figure}

\end{document}